\documentclass[10pt]{amsart}
\usepackage{a4wide}
\usepackage{times}
\usepackage{amsmath,amssymb,url,stmaryrd,enumerate,bbm,enumerate}
\usepackage[french,english]{babel}
\usepackage[latin1]{inputenc}
\usepackage{tikz-cd}
\usepackage[latin1]{inputenc}
\usepackage[T1]{fontenc}

\usepackage{xr}

\usepackage{hyperref}

\usepackage{mathrsfs}

\usepackage{yhmath}

\def\R{\mathbb{R}}

\def\Fpb{\overline{\mathbb{F}}_p}
\def\L{\mathscr{L}}
\def\La{\Lambda}
\def\A{\text{{\bf A}}}   

\def\O{\mathcal{O}} 
\def\F{\mathscr{F}}

\def\*{^\times }

\def\G{\mathscr{G}}
\def\Ga{\text{\bf G}_a}

\def\BB{\boldsymbol{B}}

\def\a{\alpha}

\def\ph{\varphi}

\def\lssi{\Longleftrightarrow}

\def\drt{\rightarrow}
\def\ldrt{\longrightarrow}
\def\Q{\mathbb{Q}}

\def\Qp{\mathbb{Q}_p}

\def\Zp{\mathbb{Z}_p}

\def\Z{\mathbb{Z}}

\def\Hom{\text{Hom}}

\def\Gal{\text{Gal}}
\def\={\! = \!}

\def\spf{\text{Spf}}

\def\E{\mathscr{E}}

\def\limp{\underset{\longleftarrow}{\text{ lim }}\;}
\def\limi{\underset{\longrightarrow}{\text{ lim }}\;}

\def\iso{\xrightarrow{\;\sim\;}}

\def\Aut{\text{Aut}}
\def\GL{\hbox{GL}}

\def\xrig{\xrightarrow}

\def\M{\mathcal{M}}

\def\X{\mathfrak{X}}
\def\GG{\Gamma}
\def\Ext{\text{Ext}}

\def\bc{\backslash}

\def\spa{\text{Spa}}

\def\Lie{\text{Lie}}

\def\Fil{\mathrm{Fil}}

\def\unp{ \big [ {\textstyle\frac{1}{p}}\big ]}

\def\sp{\mathrm{sp}}

\def\Kb{\overline{K}}

\def\<<{\langle\langle}
\def\>>{\rangle\rangle}

\def\BTrig{\mathrm{BT}^{\mathrm{rig}}}
\def\BT{\mathrm{BT}}

\newcommand{\Bmu}{\mbox{$\raisebox{-0.59ex}
  {$l$}\hspace{-0.18em}\mu\hspace{-0.88em}\raisebox{-0.98ex}{\scalebox{2}
  {$\color{white}.$}}\hspace{-0.416em}\raisebox{+0.88ex}
  {$\color{white}.$}\hspace{0.46em}$}{}}

\newcommand*{\longhookrightarrow}{\ensuremath{\lhook\joinrel\relbar\joinrel\rightarrow}}

 \DeclareMathSymbol{B}{\mathalpha}{operators}{`B}

\author{Laurent Fargues}
\address{Laurent Fargues, CNRS, Institut de Math\'ematiques de Jussieu, 4 place Jussieu 75252 Paris}
\email{laurent.fargues@imj-prg.fr}
\thanks{L'auteur a bénéficié du support  du projet ERC Advanced grant 742608 "GeoLocLang". }

\begin{document}

\title{Groupes analytiques rigides $p$-divisibles II}

\date{\today}

\maketitle

\renewcommand\labelitemi{\textbullet}

\newtheorem{theo}{Théorème}[section]
\newtheorem*{theon}{Théorème}

 \newtheorem{prop}[theo]{Proposition}

\newtheorem{coro}[theo]{{Corollaire}}

\newtheorem{lemme}[theo]{Lemme}
\newtheorem{question}[theo]{Question}

\newtheorem{defi}[theo]{Définition}
\newtheorem{exem}[theo]{Exemple}
\newtheorem{conj}[theo]{Conjecture}

\newtheorem{rema}[theo]{Remarque}

\renewcommand\thefootnote{}

\renewcommand{\thefootnote}{\arabic{footnote}}

\markright{Groupes analytiques rigides $p$-divisibles II}

\selectlanguage{french}

\begin{abstract}
Soit $K$ un corps $p$-adique. On continue de développer la théorie des groupes analytiques rigides $p$-divisibles sur $K$. On explique par exemple comment retrouver la catégorie des espaces de Banach-Colmez à partir des groupes analytiques rigides $p$-divisibles \og en niveau fini\fg{} sans espaces perfectoïdes. On établit ensuite des résultats sur les familles de groupes analytiques rigides $p$-divisibles. Cela nous permet de démontrer un théorème de \og minimalité\fg{} au sens de la géométrie birationnelle des modèles entiers des espaces de Rapoport-Zink non-ramifiés.
\end{abstract}

\selectlanguage{english}

\begin{abstract}
Let $K$ be a $p$-adic field. We continue to develop the theory of rigid analytic $p$-divisible groups over $K$. For example, we explain how to find back the category of Banach-Colmez spaces from rigid analytic $p$-divisible groups "in finite level" without perfectoid spaces. We then establish some results about families of rigid analytic $p$-divisible groups. This allows us to prove a "minimality" result in the sense of birationnal geometry 
for integral models of unramified Rapoport-Zink spaces.   
\end{abstract}

\selectlanguage{french}

\section*{Introduction}

Le but de cet article est double. On commence par rappeler et compléter les résultats de \cite{GroupesAnalytiquesI} sur les groupes analytiques rigides $p$-divisibles. On développe ensuite une théorie des familles de groupes analytiques rigides $p$-divisibles sur des bases qui sont des espaces rigides. On applique  cela afin de montrer un résultat de \og minimalité\fg{} des modèles entiers d'espaces de Rapoport-Zink non-ramifiés (\cite{RZ}) au sens de la géométrie birationnelle.

\subsection*{Description des différentes sections}

La section \ref{sec1} est consacrée à des rappels sur les groupes analytiques rigides $p$-divisibles. On en profite pour  \og nettoyer \fg{} les notations de \cite{GroupesAnalytiquesI} afin de rendre le texte plus clair.
\\

Dans la section \ref{sec2} on regroupe des résultats conséquences de \cite{GroupesAnalytiquesI}, notamment à la vue du théorème de classification de Scholze-Weinstein (\cite{ScholzeWeinstein}) apparu après et qui utilise les résultats de \cite{GroupesAnalytiquesI}. 

Tout d'abord on explique comment associer à toute réseau Galois invariant $\rho$ dans une représentation de Hodge-Tate à poids dans $\{0,1\}$ un groupe analytique rigide $G(\rho)$. En général ce groupe \og n'a pas bonne réduction\fg{} et {\it on caractérise  géométriquement les représentations $\rho$ semi-stables} (prop. \ref{prop:caracterisation semi stable}): 
$$
\rho\text{ semi-stable } \lssi
\begin{cases}
G(\rho)^0 \simeq \mathring{\BB}^d , d=\dim G(\rho) \\
\text{l'action de Galois sur }
\pi_0^{\text{géo}} (G(\rho)) \text{ est non-ramifiée}.
\end{cases}
$$
Cela nous permet de donner une infinité d'exemples de $\Qp$-espaces rigides isomorphes à des boules ouvertes sur $\mathbb{C}_p$ mais pas sur toute extension de degré fini de $\Qp$. Ce type de méthode permet également de généraliser l'exemple de Scheinder-Teitelbaum (\cite{SchneiderTaitelbaumFourier}) , cf. exemple \ref{exem:ST}.
\\

Dans la section \ref{sec:sec3} on montre que {\it la catégorie des espaces de Banach-Colmez (\cite{Colmez2}) \og descend en niveau fini \fg{}}. Plus précisément, on montre (théo. \ref{theo:BC}) que les deux catégories suivantes sont équivalentes:
\begin{enumerate}
\item 
La sous-catégorie des espaces de Banach-Colmez formée des extensions d'un $C$-espaces vectoriel de dimension finie par un $\Qp$-espace vectoriel de dimension finie 
\item  La catégorie dont les objets sont les $C$-groupes analytiques rigides $p$-divisibles
munie des morphismes qui sont les classes d'équivalence 
de {\it quasi-morphismes} à isogénie près.
\end{enumerate}

Cette notion de quasi-morphisme (déf. \ref{defi:quasimorphismes})
est inspirée de la notion analogue en théorie géométrique des groupes et de la théorie des quasi-logarithmes des groupes formels. 
\\

La section \ref{sec:sec4} est préparatoire dans l'optique de démontrer le théorème principal \ref{theo:minimal}. Il s'agit simplement d'étendre la notion de groupe analytique rigide $p$-divisible à des bases plus générales (déf. \ref{defi:famille}) et d'étendre le critère de \cite{GroupesAnalytiquesI} qui caractérise les fibres génériques de groupes formels $p$-divisibles en termes de boules (théo. \ref{theo:equivalence entre groupes formels rigides et entiers}).
\\

La section \ref{sec:sec5} contient le théorème principal de cet article. 
Une des motivations est de {\it comprendre comment retrouver les modèles entiers des espaces de Rapoport-Zink à partir d'une donnée additionnelle}. Un point de vue là dessus est donné 
dans \cite{Lourenco} (coro. 6.6) et \cite{ScholzeBerkeley} (théo. 25.4.1) où la donnée additionnelle consiste en la perfection de la fibre spéciale réduite.  On prend ici un point de vue différent. Il s'agit également d'un point de vue assez orthogonal aux résultats d'extensions de \cite{VasiuZinkPurity}.

Soit donc $\M$ un espace de Rapoport-Zink de type PEL non-ramifié sur $\spf (\breve{\Z}_p)$ de fibre générique $\M_\eta$ (\cite{RZ}). On note $\O^+$ le sous-faisceau de $\O_{\M_\eta}$ des fonctions bornées par $1$. Notons $\E^+$ le $\O^+$-module localement libre sur $\M_\eta$ obtenu à partir de l'algèbre de Lie de la déformation universelle sur $\M$.

\begin{theon}[théo. \ref{theo:minimal}]
Soit $\X$ un $\spf ( \breve{\Z}_p)$-schéma formel localement formellement de type fini formellement lisse. Soit $f:\X_\eta\drt \M_\eta$ un morphisme de $\breve{\Q}_p$-espaces rigides. Sont équivalents:
\begin{enumerate}
\item $f$ s'étend en un morphisme $\X\drt \M$
\item le $\O_{\X_\eta}^+$-module localement libre $f^* \E^+$ est localement libre sur $|\X|$ via $\sp:|\X_\eta|\drt |\X|$.
\end{enumerate}
En particulier le couple $(\M_\eta, \E^+)$ détermine complètement le modèle entier $\M$.
\end{theon}

On renvoie au corollaire \ref{coro:caracterisation concrete} pour une version peut-être plus \og concrète\fg{} de que signifie le théorème précédent.

Par exemple, si $\X$ est lisse, i.e. $p$-adique, quasi-compact, d'après Raynaud $f$ est induit par un morphisme de schémas formels $g$
$$
\begin{tikzcd}
\widetilde{\X} \ar[d] \ar[r,"g"] & \M \\
\X \ar[ru, dashed,"\exists ?"']
\end{tikzcd}
$$
où $\widetilde{\X}\drt \X$ est un éclatement formel admissible et où on peut supposer $\widetilde{\X}$ normal. Notons $H$ la déformation universelle sur $\M$, un groupe $p$-divisible.
 Le théorème précédent dit alors qu'on peut {\it descendre $g$ le long de l'éclatement} si et seulement si $g^*\Lie\, H$ est libre au dessus d'un recouvrement de $\widetilde{\X}$ provenant d'un recouvrement de $\X$ si et seulement $g^*\Lie\, H$ descend en un fibré sur $\X$.

L'outil principal de la preuve du théorème \ref{theo:minimal} est le théorème d'annulation de Bartenwerfer (\cite{BartenwerferAnnulation})
$$
H^1 (\BB^n_K,\O^+)=0
$$
pour une boule fermée sur $K|\Qp$. Cela permet de montrer (cf. prop. \ref{prop:extension groupes sur boule} et \ref{prop:extension quasi morphismes sur boule}) que si $K|\breve{\Q}_p$, un morphisme
$$
f:\BB^n_K \ldrt \M_\eta
$$
s'étend en un morphisme entier $\widehat{\A}^n_{\O_K}\drt \M$ si et seulement si $f^*\E^+$ est localement libre sur $|\widehat{\A}^n_{\O_K}|$. Dès lors la preuve du théorème \ref{theo:minimal} se décompose en les étapes suivantes:
\begin{itemize}
\item On montre grâce au résultat précédent que, via $f:\X_\eta\drt \M_\eta$, les fibres de Milnor de $\X\drt \spf (\breve{\Z}_p)$, qui sont des boules ouvertes, s'envoient sur les fibres de Milnor de $\M\drt \spf (\breve{\Z}_p)$.
\item On en déduit que $|f|$ se factorise en une application $|\X|\drt |\M|$.
\item On montre que cette application est continue en utilisant que la restriction de $\omega$ à toute composante irréductible de $\M_{red}$ est ample.
\end{itemize}

Notons enfin que dans le théorème \ref{theo:description des points} on donne une description explicite de $\M(\Fpb)$ à partir du couple $(\M_\eta,\E^+)$.

\section{Rappels de \cite{GroupesAnalytiquesI}}\label{sec1}
Soit $K|\Qp$ un corps valué complet. On note $\overline{K}$ une cloture algébrique de $K$ et $C=\widehat{\Kb}$.

\begin{defi}
Un $K$-groupe  analytique rigide $p$-divisible est un $K$-groupe analytique rigide commutatif $G$
tel que la multiplication par $p$, $G\xrig{\times p} G$, 
\begin{enumerate}
\item est topologiquement nilpotente au sens où si $U$ et $V$ sont deux voisinages affinoïdes de $0$ alors pour $n\gg 0, p^n U\subset V$.
\item est un morphisme fini surjectif.
\end{enumerate}
On note $\BTrig_K$ la catégorie correspondante.
\end{defi}

Un tel groupe est une extension (\cite{GroupesAnalytiquesI} prop. 16)
$$
0\ldrt G[p^\infty]\ldrt G\xrig{\ \log_G\ } \Lie (G)\otimes  \Ga\ldrt 0.
$$
Rappelons (\cite{GroupesAnalytiquesI} coro. 17) qu'il y a une équivalence entre $\BTrig_K$ et la catégtorie des triplets $(\La,W,\a)$ où $\La$ est une représentation continue de $\Gal (\Kb |K)$ à valeurs dans un $\Zp$-module libre de rang fini, $W$ est un $K$-espace vectoriel de dimension finie et 
$\a:W_C(1)\drt \La_C$ est un morphisme $C$-linéaire compatible à l'action de Galois.

À un tel triplet on associe le groupe $G$ obtenu par descente galoisienne (\cite{GroupesAnalytiquesI} sec. 4.1)  à partir du $C$-groupe analytique rigide obtenu par produit fibré
$$
\begin{tikzcd}
0\ar[r] &   \La\otimes \Qp/\Zp\ar[d, equal] \ar[r] & G_C \ar[d] \ar[r, "\log_{G_C}"] & W_C \otimes \Ga \ar[d, "\a(-1)"] \ar[r] & 0 \\
0 \ar[r] & \La(-1)\otimes \mu_{p^\infty} \ar[r] & \La(-1)\otimes \widehat{\mathbb{G}}_m^{rig} \ar[r, "\mathrm{Id}\otimes \log"] & \La_C(-1) \otimes \Ga \ar[r] & 0
\end{tikzcd}
$$
On a alors $\La= T_p (G[p^\infty])$ et $W=\Lie\, G$.
\\

Soit $\BT_{\O_K}^{\mathrm{f}}$ la catégorie des groupe formels $p$-divisibles sur $\O_K$.
Rappelons  qu'il y a un foncteur fibre générique pleinement fidèle (\cite{GroupesAnalytiquesI} prop. 29)
$$
\BT_{\O_K}^{\mathrm{f}}\longhookrightarrow \BTrig_K
$$
d'image essentielle les groupes $G$ tels que $$G\simeq \mathring{\BB}^d_K$$ en tant que $K$-espace rigide (\cite{GroupesAnalytiquesI} théo. 6.1).
\\
D'après \cite{ScholzeWeinstein} (qui font converger l'approche de la remarque 10 de \cite{GroupesAnalytiquesI}) il s'agit d'une équivalence lorsque $K$ est algébriquement clos. Le triplet $(\La,W,\a)$ associé à $\G^{\mathrm{rig}}$, $\G\in \BT_{\O_K}^f$, est $$\big (T_p (\G), \Lie\, \G \unp,\a_{\G^D}^*(1)\big )$$ où  $\a_{\G^D}$ désigne l'application des périodes de Hodge-Tate du dual de Cartier $\G^D$ (\cite{GroupesAnalytiquesI} théo. 3.4).

\section{Groupes  analytiques rigides $p$-divisibles, représentations Galoisiennes et formes tordues de boules ouvertes}\label{sec2}

Dans cette section $K$ est de valuation discrète à corps résiduel parfait. 
Soit $$\rho: \Gal (\Kb|K)\drt \GL (\La)$$
une représentation Galoisienne à valeurs dans  un $\Zp$-module libre de rang fini tel que $\rho\unp$ soit de Hodge-Tate à poids dans $\{0,1\}$. Notons 
$$
G(\rho)
$$
le groupe analytique rigide associé au triplet $(\La, \La_C(-1)^{\Gal (\Kb|K)}, \a)$
avec $\a$ l'inclusion canonique. Cela définit une correspondance fonctorielle
$$
\rho \mapsto G(\rho).
$$
Ainsi, à chaque telle représentation on peut associer un groupe analytique rigide qui n'a pas forcément \og bonne réduction\fg{} contrairement au cas cristallin. On peut par exemple caractériser géométriquement les représentations semi-stables en termes géométriques.

\begin{prop}\label{prop:caracterisation semi stable}
La représentations $\rho\unp$ est semi-stable si et seulement si $G(\rho)^0 \simeq \mathring{\BB}^d_K$ et l'action de $\Gal (\Kb |K)$ sur $\pi_0 (G)(\Kb)$ est non-ramifiée.
\end{prop}
\begin{proof}
Soit $A=(D,\ph,N,\Fil^\bullet D_K)$ un $(\ph,N)$-module filtré faiblement admissible tel que
$\Fil^0 D_K\\ =D_K$ et $\Fil^2 D_K=0$. Notons $D_{>0}$ la \og partie de pente $>0$\fg{} dans l'isocristal $(D,\ph)$. Puisque $(D,\ph)$ est à pente dans $[0,1]$ et $N:(D,\ph)\drt (D,p\ph)$, $N .D_{>0}=0$. On vérifie aussitôt que $B=(D_{>0}, \ph, \Fil^\bullet D_K\cap D_{>0}\otimes_{K_0} K)$ est admissible. On en déduit que $A$ est une extension de l'isocristal filtré $(D_0,\ph)$ muni de la filtration triviale $\Fil^0 D_{0, K}=D_{0, K}$ et $\Fil^1 D_{0, K}$ par $B$. 

De cette analyse il est aisé de déduire que la catégorie des représentations semi-stables à poids de Hodge-Tate dans $\{0,1\}$ coïncide avec celle des extensions de représentations non-ramifiées par les représentations cristallines à poids de Hodge-Tate dans $\{0,1\}$ dont l'isocristal associé est à pentes dans $]0,1]$.
\end{proof}

Notons le corollaire suivant.

\begin{coro}
A chaque réseau Galois invariant $\rho$ dans une représentation de Hodge-Tate à poids dans $\{0,1\}$ qui n'est pas potentiellement semi-stable on peut associer un $K$-espace rigide $G(\rho)^0$ qui est isomorphe à une boule ouverte après extension des scalaires à $C$, mais ne l'est pas sur toute extension de degré fini de $K$.
\end{coro}

Une autre manière de construire des formes tordues de boules ouvertes du type précédent consiste à choisir $\La$ comme précédemment mais de prendre $W\subsetneq \La_C(-1)^{\Gal (\Kb|K)}$.

\begin{exem}\label{exem:ST}
Supposons que $K=\Qp$ et soit $L|\Qp$ de degré fini muni d'un plongement $\tau:L\hookrightarrow C$. Le groupe $G$ sur $\Qp$ associé au triplet $(\O_L (1), C,\a)$ avec $\a: C(1)\hookrightarrow L\otimes_{\Qp} C(1)$ induit par $\tau$ coïncide avec celui étudié par Schneider et Teitelbaum dans \cite{SchneiderTaitelbaumFourier}. On renvoie également au théorème 6.1 de \cite{BergerColmezSen} pour la construction de formes tordues d'espaces homogènes quotients de boules ouvertes. 
\end{exem}

%
%
%

\section{Quasi-moprhismes et déperfectoïdisation des espaces de Banach-Colmez}
\label{sec:sec3}

Rappelons la définition suivante (\cite{GroupesAnalytiquesI} sec. 8).

\begin{defi}\label{defi:quasimorphismes}
Pour $G_1$ et $G_2$ deux groupes analytiques rigides $p$-divisibles 
\begin{enumerate}
\item 
On note
$$
Q\Hom (G_1,G_2)
$$
le groupe des morphismes de $K$-espaces rigides $f:G_1\drt G_2$ tels que 
$$
f(x+y)-f(x)-f(y): G_1\times G_1\ldrt G_2
$$
est \og borné\fg{} i.e. d'image quasi-compacte dans $G_2$.
\item On note 
$$
Q\Hom (G_1,G_2)/{{\sim} } 
$$
le quotient par le sous-groupe des morphismes $G_1\drt G_2$ d'image \og bornée\fg{}.
\end{enumerate}
\end{defi}

Cette notion est inspirée de la notion de {\it quasi-logarithme des groupes formels} $p$-divisibles (\cite{Fontaine5} sec. 6.4 et \cite{Katz2})  mais également de la notion de {\it quasi-morphisme en théorie géométrique des groupes}. On a par exemple, du point de vue de la théorie des groupes classiques
\begin{eqnarray*}
Q\Hom (\Z,\Z)/{\sim} & \iso &  \R \\
{[} f] & \longmapsto & \underset{n\drt +\infty}{\lim} \tfrac{f(n)}{n}
\end{eqnarray*}
d'inverse $x\mapsto \big [ n\mapsto [nx] \big ]$.
On va utiliser des formules $p$-adiques analogues du type \og $\underset{n\drt +\infty}{\lim} p^n f(p^{-n}x)$\fg{}.
\\

Soit $\mathfrak{a}\subsetneq \O_K$ un idéal principal non nul. Rappelons que, par rigidité des morphismes entre groupes $p$-divisibles, le membre de gauche dans la formule de  l'énoncé du lemme qui suit 
ne dépend pas du choix d'un tel $\mathfrak{a}$.

\begin{lemme}\label{lemme:quasi morphismes et morphismes modulo p}
Soient $\G_1$ et $\G_2$ deux groupes formels $p$-divisibles sur $\O_K$. On  a une identification 
$$
\Hom (\G_1\otimes \O_K /\mathfrak{a} , \G_2\otimes \O_K /\mathfrak{a})\unp= Q\Hom (\G_1^{rig} ,\G_2^{rig} ) /{\sim} \unp.
$$
\end{lemme}
\begin{proof}
Il suffit d'associer à un morphisme $\bar{f}:\G_1\otimes \O_K /\mathfrak{a} \drt \G_2\otimes \O_K /\mathfrak{a}$, la classe du morphisme $f^{rig}$ où $f:\G_1\drt \G_2$ est n'importe quel relèvement comme morphisme de $\O_K$-schémas formels.
\end{proof}

On prend le point de vue des faisceaux pro-étales sur les espaces de Banach-Colmez (\cite{Colmez2}, \cite{Courbe} chap. 3, \cite{AC-BC}).

\begin{theo}\label{theo:BC}
Supposons $K=C$ algébriquement clos. 
Le foncteur revêtement universel $$G\mapsto \underset{\times p}{\limp} G^\diamond$$  induit une équivalence entre
\begin{enumerate}
\item 
 La catégorie des groupes analytiques rigides $p$-divisibles munis des classes d'équivalence de quasi-morphismes à isogénie près.
 \item  La catégorie des espaces de Banach-Colmez (\cite{Colmez2}, \cite{Colmez3}) extension d'un $C$-espace vectoriel de dimension finie par un $\Qp$-espace vectoriel de dimension finie.
 \end{enumerate}
\end{theo}
\begin{proof}
Le foncteur est défini de la façon suivante. Si $f:G_1\drt G_2$ est un quasi-morphisme alors on lui associe 
\begin{eqnarray*}
\underset{\times p}{\limp} G^\diamond_1 & \ldrt & \underset{\times p}{\limp} G^\diamond_2 \\
(x_n)_{n\geq 0} & \longmapsto & \big ( \underset{k\drt +\infty}{\lim} p^k f(x_{n+k})\big )_{n\geq 0}.
\end{eqnarray*}
La signification de l'expression précédente est que si $(R,R^+)$ est une $C$-algèbre affinoïde perfectoïde alors $G_1(R)$ et $G_2 (R)$ sont naturellement des $\Zp$-modules topologiques, une base de voisinages de $0$ étant donnée par les $U(R,R^+)$ avec $U$ un sous-groupe affinoïde. Dès lors la limite précédente a un sens puisque pour un tel $U$, $U(R,R^+)$ est $p$-adiquement complet et $f(px)-pf(x):G_1\drt G_2$ est \og borné\fg{} et donc à valeurs dans un tel sous-groupe affinoïde de $G_2$. Le foncteur est bien à valeurs dans la catégorie annoncée car pour $G\in \BT_C^{rig}$ 
$$
0\ldrt \underline{V_p(G[p^\infty])} \ldrt \underset{\times p}{\limp} G^\diamond \ldrt \Lie\, G\otimes \mathbb{G}_a^{\diamond}\ldrt 0.
$$
La surjectivité essentielle vient de la classification de tels espaces de Banach-Colmez i.e. le calcul (\cite{Colmez2}, \cite{AC-BC})
$$
\Ext^1 ( W\otimes \Ga^{\diamond}, \underline{V}) = \Hom (W(-1),V_C)
$$
donné par l'extension universelle
$$
0\ldrt \underline{\Qp.t} \ldrt \BB^{\ph=p} \xrig{\ \theta\ } \mathbb{G}_a^\diamond\ldrt 0.
$$
Il reste à vérifier la pleine fidélité. 
Soient donc $G_1,G_2\in \BT_C^{rig}$. 
Utilisant Scholze-Weinstein on peut écrire 
$$
G_1=\G_1^{rig}\oplus \Ga^{n_1} \oplus (\Qp/\Zp)^{m_1}  \text{ et } G_2=\G_2^{rig} \oplus \Ga^{n_2} \oplus (\Qp/\Zp)^{m_2}
$$
avec $\G_1,\G_2\in \BT_{\O_C}^f$. Traitons d'abord le cas $n_1=m_1=n_2=m_2=0$. On utilise le lemme \ref{lemme:quasi morphismes et morphismes modulo p}. Pour $\G\in \BT_{\O_C}^{\mathrm{f}}$ soit $\widetilde{\G}\in \BT_{\O_{C^\flat}}$ tel que 
$$
\G\otimes \O_C/p = \widetilde{\G}\otimes \O_{C^\flat}/\underline{p}. 
$$
Alors, 
$
\underset{\times p}{\limp} \G^{rig,\diamond}
$
est représenté par le $C^\flat$ groupe perfectoïde
$
\widetilde{\G}^{\, rig,1/p^\infty}
$ (cf. chapitre 3 de \cite{Courbe}). Le résultat se déduit alors de l'égalité 
$$
\Hom (\widetilde{\G}_1^{\, rig,1/p^\infty}, \widetilde{\G}_2^{\, rig,1/p^\infty}) = \Hom ( \G_1\otimes \O_C/p,\G_2\otimes \O_C /p)\unp
$$ 
obtenu en envoyant un morphisme $f:\widetilde{\G}_1^{\, rig,1/p^\infty}\drt \widetilde{\G}_2^{\, rig,1/p^\infty}$ sur $p^{-N}\times $ le morphisme
$$
(p^N f)^*: \O\big (\widetilde{\G}_1^{\, rig,1/p^\infty}\big )^+/\underline{p} \drt \O\big ( \widetilde{\G}_1^{\, rig,1/p^\infty}\big )^+ /\underline{p}
$$ 
restreint à $\O(\G_1\otimes \O_C/p)$
pour $N\gg 0$. 

Traitons maintenant du cas de $G_1=\Qp/\Zp$ et $G_2$ quelconque. Choisissons une section ensembliste $s$ de la projection $\Qp\drt \Qp/\Zp$. Si $f:\underline{\Qp}\drt \underset{\times p}{\limp} G^\diamond$ on lui associe le morphisme de diamants
$$
\underline{\Qp/\Zp} \xrig{ \ s\ } \underline{\Qp} \xrig{\ f\ }\underset{\times p}{\limp} G^\diamond \xrig{\text{ proj }} G^\diamond.
$$ 
Celui-ci provient d'un morphisme $\Qp/\Zp\drt G$ qui est un quasi-morphisme car $\mathrm{proj}\circ f_{|\underline{\Zp}}$ est à valeurs dans un ouvert quasi-compact de $|G^\diamond|=|G|$. On vérifie que l'association $f\mapsto $ [la classe de ce quasi-morphisme] définit un inverse à 
$$
Q\Hom (\Qp/\Zp, G)/{\sim}\unp \drt \Hom ( \underline{\Qp}, \underset{\times p}{\limp} G^\diamond ).
$$
Le cas des quasi-morphismes de $\Ga$ vers $G\in \BT_C^{rig}$ quelconque se traite immédiatement en utilisant:
\begin{itemize}
\item tout morphisme d'espaces rigides de $\mathbb{G}_a$ vers $\G^{rig}$, $\G\in \BT_{\O_C}^{\mathrm{f}}$, est constant
\item la pleine fidélité du foncteur $(-)^\diamond$ sur les $C$-espaces rigides lisse.
\end{itemize} 
Il reste à traiter le cas des quasi-morphismes de $G$ vers $\mathbb{G}_a$. On peut supposer $G=\G^{rig}$ avec $\G\in \BT_{\O_C}^{\mathrm{f}}$. Le résultat est alors une conséquence de la proposition 32 de \cite{GroupesAnalytiquesI}.
\end{proof}

\begin{rema}
Par définition, 
tout espace de Banach-Colmez est un quotient d'un espace intervenant dans le théorème \ref{theo:BC} par un sous-$\Qp$-espace vectoriel de dimension finie. Ce théorème permet donc de reconstruire la catégorie des espaces de Banach-Colmez à partir d'objets de géométrie rigide classique non-perfectoïdes. On peut donc reconstruire la courbe (\cite{Courbe}) à partir de ces mêmes objets.
\end{rema}

\begin{exem}
On a un diagramme d'identifications
$$
\begin{tikzcd}[column sep=tiny]
0 \ar[r] & Q\Hom ( \Qp/\Zp, \Bmu_{p^\infty})/{\sim}\unp \ar[d,equal] \ar[r] & 
Q\Hom ( \Qp/\Zp, \widehat{\mathbb{G}}_m^{rig})/{\sim} \unp \ar[d, equal]\ar[r,"\log"] &
Q\Hom ( \Qp/\Zp ,\Ga)/{\sim} \unp \ar[d,equal] \ar[r] & 0  \\
0\ar[r] & \Qp t \ar[r] &
 (B^+_{cris})^{\ph=p} \ar[r, "\theta"] & C \ar[r] & 0
\end{tikzcd}
$$
où les identifications verticales sont données par $[f] \mapsto \underset{n\drt +\infty}{\lim} p^nf ( p^{-n} \,\mathrm{ mod }\, \Zp)$. Par exemple, cette formule induit un isomorphisme $Q\Hom ( \Qp/\Zp, \Bmu_{p^\infty})/{\sim}\unp \iso \Hom (\Qp/\Zp,\Bmu_{p^\infty})\unp$.        
\end{exem}

\section{Familles de groupes analytiques rigides $p$-divisibles}
\label{sec:sec4}

Nous allons maintenant étudier les familles de groupes analytiques rigides $p$-divisibles
dans l'optique de démontrer le théorème \ref{theo:minimal}.

\begin{defi}\label{defi:famille}
Un $S$-groupe analytique rigide $p$-divisible est un $S$-groupe analytique rigide commutatif lisse $G$ tel que la multiplication par $p$, $G\xrig{\times p} G$, 
\begin{enumerate}
\item est topologiquement nilpotente au sens où, localement sur $S$, si $U$ et $V$ sont deux voisinages affinoïdes de $0$ alors pour $n\gg 0, p^n U\subset V$.
\item est un morphisme fini surjectif.
\end{enumerate}
\end{defi}

Comme dans \cite{GroupesAnalytiquesI}  un tel groupe est une extension
$$
0\ldrt G[p^\infty]\ldrt G\xrig{\ \log_G\ } \Lie (G)\otimes  \Ga\ldrt 0
$$
où $\Lie (G)$ est un fibré vectoriel sur $S$ et pour tout $n\geq 1$, $G[p^n]$ est un $S$-groupe étale fini (les arguments de la section 1.5 de \cite{GroupesAnalytiquesI} s'adaptent aussitôt). 

\begin{rema}
Considérons un triplet  $(\F, \E,\a)$ où $\F$ est un $\underline{\Z_p}$-système local étale sur $S$, $\E$ un fibré vectoriel et $\a: \E(1)\drt  \F \otimes_{\underline{\Zp}} \O_S$ est un morphisme de faisceaux de $\O_S$-modules sur le site pro-étale de $S$. On peut lui associer un $S$-groupe analytique rigide $p$-divisible $G$ comme dans \cite{GroupesAnalytiquesI} avec $\F=T_p (G[p^\infty])$ et $\E=\Lie\, G$. Néanmoins il se peut à priori que tout les $S$-groupes analytiques rigides $p$-divisibles ne proviennent pas de cette construction. En tentant d'adapter les arguments de la section 3 de \cite{GroupesAnalytiquesI} on tombe sur le problème de savoir si tout élément de $\mathrm{Pic} (\BB^1_S)$ est localement trivial sur $S$, ce qui n'est pas connu même pour $S$ lisse (\cite{KerzSaitoTamme}).
\end{rema}

On dispose maintenant du résultat suivant qui est une généralisation du théorème 6.1 de \cite{GroupesAnalytiquesI}. Lorsque $K$ est de valuation discrète on dispose d'un bonne notion de schéma formel admissible normal sur $\spf (\O_K)$  (\cite{LivreIso} Annexe A2 du Chap. I).

\begin{theo} \label{theo:equivalence entre groupes formels rigides et entiers}
\begin{enumerate}
\item 
Supposons $K$ de valuation discrète et 
soit $\mathscr{S}$ un $\O_K$-schéma formel admissible normal  de fibre générique $S$. Le foncteur fibre générique induit une équivalence entre
\begin{enumerate}
\item Les groupes formels $p$-divisibles sur $\mathscr{S}$
\item Les $S$-groupes analytiques rigides $p$-divisible localement isomorphes sur $|\mathscr{S}|$ à $\mathring{\BB}^d_S$ pour un entier $d$.
\end{enumerate}
\item Supposons $K$ quelconque. Soit $\mathscr{S}=\spf ( R)$ un $\O_K$-schéma formel admissible affine tel que $R$ soit normal et de fibre générique $S$. 
 Le foncteur fibre générique induit une équivalence entre
\begin{enumerate}
\item Les groupes formels $p$-divisibles sur $\mathscr{S}$ dont l'algèbre de Lie est triviale.
\item Les $S$-groupes analytiques rigides $p$-divisibles isomorphes à $\mathring{\BB}^d_S$ pour un entier $d$.
\end{enumerate}
\end{enumerate}
\end{theo}
\begin{proof}
Il suffit de suivre la démonstration de \cite{GroupesAnalytiquesI}, théo. 6.1, pas à pas. Le point est de montrer que si $R$ est $p$-adique normal topologiquement de type fini sur $\O_K$, 
$$
f:R \llbracket x_1,\dots,x_d\rrbracket \ldrt R\llbracket x_1,\dots,x_d \rrbracket
$$
induit un morphisme étale fini $\mathring{\BB}^d_S\drt \mathring{\BB}^d_S$ avec $S=\mathrm{Sp} (R\unp)$, alors $f$ est fini localement libre. On utilise pour cela comme dans 
\cite{GroupesAnalytiquesI} le théorème 5.8 du chapitre V de \cite{Zi1}.
\end{proof}

\section{Un théorème d'extension de morphismes}
\label{sec:sec5}

\subsection{Formes tordues de boules ouvertes}

Soit $S$ un $K$-espace analytique rigide. 
Par définition, une fibration en boules ouvertes pointées est un $S$-espace rigide 
$$
B\drt S
$$
muni d'une section $S\drt B$, localement isomorphe sur $S$ à une boule ouverte $\mathring{\BB}^d$ munie de sa section nulle.

Le faisceau $\underline{\Aut} (\mathring{\BB}^d)$ s'identifie aux éléments de 
$$
\big ( (x_1,\dots,x_d) \O^+ \llbracket x_1,\dots, x_d \rrbracket \big )^d
$$
dont \og la dérivée en $(0,\dots,0)$\fg{}, obtenue en réduisant modulo $(x_1,\dots,x_d)^2$, est un élément de $\GL_d (\O^+)$. La loi de groupe est donnée par la composition de telles séries formelles. On en déduit  la proposition suivante.

\begin{prop}\label{prop:equivalence groupes formels et rigides}
Si $S$ est un $K$-espace rigide, à chaque fibration en boules ouvertes pointées sur $S$ est associé naturellement un $\O_S^+$-module localement libre de rang fini qui, après application de $-\otimes_{\O_S^+} \O_S$, redonne le fibré conormal à la section de la fibration.
\end{prop}

Nous allons maintenant utiliser le résultat suivant de Bartenwerfer (\cite{BartenwerferAnnulation}):
$$
H^1 (\BB^n_K, \O^+)=0.
$$
En d'autres termes, 
$$
\underset{\X\drt \widehat{\A}^n_{\O_K}}{\limi}  H^1 (\X,\O_X) =0
$$
où $\X\drt \widehat{\A}^n_{\O_K}$ parcourt les éclatements formels admissibles de $\widehat{\A}^n_{\O_K}$. Les arguments qui vont suivre utilisent exactement cette annulation qui est essentielle. En d'autres termes le fait que ce groupe de cohomologie soit annulé par une puissance de $p$ (\cite{BartenwerferPresqueAnnulation}) est insuffisant pour nos besoins (cf. également la discussion sur le cas perfectoïde qui suit la proposition \ref{prop:trivialite fibration boule}).

\begin{prop}\label{prop:trivialite fibration boule}
Soit $n\geq 1$ et 
$B\drt \BB^n_K$ une fibration en boules ouvertes pointées au dessus d'une boule fermée de rayon $1$. Sont équivalents:
\begin{enumerate}
\item La fibration $B\drt \BB^n_K$ est triviale.
\item Le $\O^+$-module localement libre associé sur $\BB^n_K$ est trivial. 
\end{enumerate}
\end{prop}
\begin{proof}
Le groupe $\underline{\Aut} (\mathring{\BB}^d)$ est muni d'une filtration $(\Fil^i)_{i\geq 1}$ pour laquelle il est complet et  telle que
$$
\Fil^1 /\Fil^2 \simeq \GL_d (\O^+)
$$
et pour $i\geq 2$
$$ 
\Fil^i / \Fil^{i+1} \simeq (\O^+)^{d^2}.
$$
Le résultat se déduit alors de l'annulation de $H^1 (\BB^n_K, \O^+)$.
\end{proof}

Soit $S=\spa (R,R^+)$ un espace affinoïde perfectoïde. On peut considérer l'espace adique sous-perfectoïde $\mathring{\BB}^d_S$ et donc parler de fibrations en boules ouvertes pointées au dessus de $S$ comme précédemment. Soit $\varpi \in R$ une pseudo-uniformisante et considérons la boule \og fermée\fg{} $\BB_S^d(|\varpi|)=\spa ( R\langle \frac{T}{\varpi}\rangle , R \langle \frac{T}{\varpi}\rangle^+ ) \subset \mathring{\BB}^d_S$. Tout automorphisme pointé de $\mathring{\BB}^d_S$ laisse invariant la boule fermée précédente et induit un automorphisme pointé de $\BB_S^d(|\varpi|)$. On peut alors utiliser la même méthode de preuve que dans la proposition \ref{prop:trivialite fibration boule} couplé à la presque annulation de $H^1(S,\O^+)$. On en déduit que si $B\drt S$ est une fibration en boules ouvertes pointées telle que le $\O^+$-module associé soit trivial alors 
$$
B=\bigcup_{n\geq 1} \BB^d_S ( 0, |\varpi|^{1/p^n})
$$
est une union croissante de fibrations triviales en boules fermées. Néanmoins on ne peut conclure que la fibration en boules ouvertes est triviale en général sauf bien sûr si $S$ est totalement discontinu (\cite{ScholzeCohomologyDiamonds}).

\subsection{Extension de groupes et de morphismes}

\begin{prop}\label{prop:extension groupes sur boule}
Soit $G\drt \BB^n_K$ un groupe analytique rigide $p$-divisible localement isomorphe à $\mathring{\BB}^d$ sur $\BB^n_K$ et tel que le $\O^+$-module localement libre associé sur $\BB^n_K$ soit trivial. Alors $G$ s'étend en un groupe formel $p$-divisible sur $\widehat{\A}^n_{\O_K}$.
\end{prop}
\begin{proof}
C'est une conséquence des propositions \ref{prop:trivialite fibration boule}, \ref{prop:equivalence groupes formels et rigides} et du théorème \ref{theo:equivalence entre groupes formels rigides et entiers}.
\end{proof}

La définition \ref{defi:quasimorphismes} des quasimorphismes et de leur classe d'équivalence s'étend aussitôt au cas des groupes analytiques rigides $p$-divisibles sur un espace rigide quasi-compact quasi-séparé.

\begin{prop}\label{prop:extension quasi morphismes sur boule}
Soient $G_1$ et $G_2$ satisfaisant les hypothèses de la proposition \ref{prop:extension groupes sur boule}. Les sections globales du $S$-faisceau $\F$ associé au préfaisceau 
$$
U\mapsto  Q\Hom ( G_{1 |U},G_{2|U}) / \sim \unp
$$
sont
$$
H^0(S,\F) = Q\Hom (G_1,G_2) / \sim \unp .
$$
\end{prop}
\begin{proof}
Notons $S=\BB^n_K$.
On a $\F=\mathscr{G}\unp /\mathscr{H}\unp $ où $\mathscr{G}$ est le $S$-faisceau des morphismes
$f:G_1\drt G_2$ vérifiant $f(0)=0$ et tels que $f(x+y)-f(x)-f(y)$ soit \og d'image bornée\fg{} et $\mathscr{H}$ est le sous-faisceau des morphismes \og d'image bornée \fg{} de $G_1$ vers $G_2$. Fixons des trivialisations $\mathring{\BB}^{d_1}_S\iso G_1$ et $\mathring{\BB}^{d_2}_S\iso G_2$. Dès lors $\mathscr{H}\unp$ est obtenu par localisation par inversion de $p$ à partir du faisceau 
$$
\big ( (p x_1,\dots,p x_{d_2}) \O_S^+\llbracket px_1,\dots, px_{d_2} \rrbracket \big )^{d_1}
$$ 
où $x_1,\dots, x_{d_2}$ sont les coordonnées sur $\mathring{\BB}^{d_2}_S$ et la loi de groupe est déduite de celle de $G_2$. Ce faisceau possède une filtration $(\Fil^i)_{i\geq 1}$, pour laquelle il est complet de gradués isomorphes à 
$$
\big ( (\O_S^+)^{d_1 d_2}, + \big ).
$$
On peut alors de nouveau appliquer le théorème de Bartenwerfer pour conclure que 
$H^1 (S, \mathscr{H}\unp)=0$. Cela démontre le résultat.
\end{proof}

\subsection{Minimalité des modèles entiers de Rapoport-Zink}

Supposons $K$ de valuation discrète 
Soit $\X$  un $\O_K$-schéma formel localement formellement de type fini formellement lise. Notons $\X_\eta$ sa fibre générique munie de son morphisme de spécialisation
$$
\sp: |\X_\eta| \twoheadrightarrow |\X|
$$
 On a alors (\cite{deJong1} théo. 7.4.1)
$$
\X= ( |\X |, \sp_* \O_{\X_{\eta}}^+).
$$
Si $\E$ est un fibré vectoriel sur $\X$ on note de façon abusive
$$
\E\otimes_{\O_{\X}} \O_{\X_\eta}^+:= \sp^{-1}\E\otimes_{\sp^{-1} \O_{\X}} \O_{\X_\eta}^+.
$$
Les foncteurs $(-)\otimes_{\O_X}\O_{\X_\eta}^+$ et $\sp_*$ induisent alors des équivalences entre
\begin{enumerate}
\item Les fibrés vectoriels sur $\X$
\item Les $\O_{\X_\eta}^+$-modules localement libres sur $|\X|$ i.e. libres au dessus d'un recouvrement ouvert de $\X_\eta$ image réciproque par l'application de spécialisation $\sp$ d'un recouvrement ouvert de $\X$.
\end{enumerate}
\

Soit maintenant $\mathbb{H}$ un groupe $p$-divisible sur $\Fpb$. Considérons l'espace de Rapoport-Zink (\cite{RZ})
$$
\M
$$
des déformations par quasi-isogénies de $\mathbb{H}$. Il s'agit d'un $\spf ( \breve{\Z}_p)$-schéma formel localement formellement de type fini formellement lisse.
\\

Nous aurons besoin du dévissage suivant au cas des groupes formels.
Il s'agit d'un énoncé du type \og coordonnées canoniques de Serre-Tate\fg{} (\cite{KatzCoordonnees}, \cite{DeligneCoordonnees}, \cite{Har4} pour le cas des espaces de Lubin-Tate). 
 Notons $\mathbb{H}^0$ la composante connexe neutre de $\mathbb{H}$ et $\M_0$ l'espace de Rapoport-Zink associé. On note $H_0\drt \M_0$ le groupe formel $p$-divisible universel au dessus de $\M_0$. On le voit comme un $\M$-schéma formel localement isomorphe à $\M\times \spf (\breve{\Z}_p \llbracket x_1,\dots, x_d\rrbracket )$.

\begin{prop}\label{prop:devissage connexe}
Notons $h$ la hauteur de la partie étale de $\mathbb{H}$.
Il y a un isomorphisme
$$
\M \simeq H_0^h\underset{GL_h (\Zp)}{\times} \GL_h (\Qp). 
$$
\end{prop}
\begin{proof}
C'est une conséquence du fait suivant (\cite{ChaiExtensions} prop. 2.9). Soit $H_1$ un groupe $p$-divisible formel sur un schéma $S$ annulé par une puissance de $p$ et $H_2$ un groupe $p$-divisible étale sur $S$. Les extensions de $H_2$ par $H_1$ sont rigides et on peut regarder le faisceau fppf des extensions $\underline{Ext} (H_2, H_1)\drt S$. Il est représenté par 
$$
\underline{Ext} (H_2,H_1) \simeq T_p (H_2)^*\otimes_{\Zp} H_1.
$$
Plus précisément, toute extension est obtenue par poussé en avant via un morphisme $\underline{T_p(H_2)}\drt H_1$ à partir de l'extension de faisceaux pro-étales
$$
0 \ldrt \underline{T_p(H_2)}\ldrt \underline{V_p(H_2)}\ldrt H_2\ldrt 0.
$$
\end{proof}

\begin{theo}\label{theo:minimal}
Soit $\X$ un $\spf (\breve{\Z}_p)$-schéma formel localement formellement de type fini et formellement lisse. Un morphisme
$$
f:\X_\eta\drt \M_\eta
$$
 s'étend en  un morphisme
$$
\X\drt \M
$$ 
si et seulement si le $\O_{\X_\eta}^+$-module $f^*\big ( \Lie (H)\otimes_{\O_{\M}} \O_{\M_\eta}^+\big )$ est localement libre sur $|\X|$. 
\end{theo}
\begin{proof}
Dans cette preuve on considère les fibres génériques commes des espaces de Berkovich.
Commençons par dévisser le résultat au cas où $\mathbb{H}$ est connexe grâce à la proposition \ref{prop:devissage connexe} dont on reprend les notations. Puisque $\X$ est formellement lisse,
$$
\pi_0 (\X_\eta) = \pi_0(\X).
$$
On peut donc supposer donné un morphisme 
$$
\X_\eta \ldrt H_{0,\eta}\ldrt \M_{0,\eta}
$$
Supposons donc démontré que le composé $\X_\eta\drt \M_{0,\eta}$ s'étend en un morphisme 
$\X\drt \M_{0}$. Le résultat se déduit alors du fait que localement sur $\M_{0}$, $H_0\simeq \M_0\times \spf (\breve{\Z}_p \llbracket x_1,\dots,x_d \rrbracket )$. 
\\

On suppose donc $\mathbb{H}$ connexe.
Soit $x\in \X_\eta$ de corps résiduel (complété) $k(x)$. Notons $K=k(x)$. On a alors, si $\sp:|\X_\eta|\drt |\X|$ et $\sp_x: \X_\eta \hat{\otimes} k(x)\drt |\X\hat{\otimes} \O_{k(x)}|$,  
$$
\mathring{\BB}^n_K \simeq \sp_x^{-1} (\sp_x (x))\twoheadrightarrow \sp^{-1} (\sp (x))
$$ 
où le membre de droite est un sous-ensemble de $|\X_\eta|$ et celui de gauche un $K$-espace de Berkovich. Montrons que l'application composée 
$$
\sp^{-1} (sp(x)) \hookrightarrow |\X_\eta| \xrig{\ f\ } |\mathfrak{\M}_\eta |\xrig{\ \sp\ } |\mathfrak{\M}|
$$
est constante. Quitte à étendre $K$, auquel cas $|\mathring{\BB}^n_K| \drt \sp^{-1}(\sp(x))$ est encore surjectif, on peut supposer que $|K^\times|$ contient $p^{\Q}$. Soient $0<\rho_1<\rho_2<1$ avec $\rho_1,\rho_2\in |K^\times|$ deux rayons. D'après le théorème  \ref{theo:equivalence entre groupes formels rigides et entiers}, les propositions  \ref{prop:extension groupes sur boule} et \ref{prop:extension quasi morphismes sur boule} le diagramme
$$
\BB^n_K(\rho_1)\hookrightarrow \BB^n_K (\rho_2) \ldrt \M_\eta
$$
s'étend en un diagramme
$$
\widehat{\A}^n_{\O_K} (\rho_1) \drt \widehat{\A}^n_{\O_K} (\rho_2) \ldrt \M 
$$
où si $\rho= |a|$ par définition $\widehat{\A}^n_{\O_K} (\rho):= \spf (\O_K\langle \frac{T}{a}\rangle )$. Le morphisme $$\widehat{\A}^n_{\O_K} (\rho_1) \drt \widehat{\A}^n_{\O_K} (\rho_2)$$ a pour image l'origine de $\A^n$ en fibre spéciale. En faisant varier les rayons $\rho_1$ et $\rho_2$ on en déduit la constance de l'application cherchée.

On a donc démontré l'existence d'une application $g:|\X|\drt |\M|$ s'inscrivant dans un diagramme
$$
\begin{tikzcd}
 {|}\X_\eta | \ar[r, "|f|"]\ar[d, twoheadrightarrow, "\sp"] & {|} \M_\eta | \ar[d, twoheadrightarrow , "\sp"] \\
{|}\X|\ar[r, "g"] & {|}\M|.
\end{tikzcd}
$$
Montrons que cette application $g$ est continue. Soit $\omega= (\det \Lie H)^*$ comme fibré $J(\Qp)$-équivariant sur le schéma formel $\M$. Soit $\mathcal{V}\subset \M$ un ouvert quasi-compact. D'après le théorème d'uniformisation de Rapoport-Zink (\cite{RZ} théo. 6.24) couplé au même type de résultat pour les variétés de Shimura, la restriction de $\omega$ à toute composante irréductible de $\M_{red}$ est ample. 

Il existe un sous-groupe discret cocompact $\Gamma\subset J(\Qp)$, agissant sans point fixes sur $\M$, tel que 
\begin{itemize}
\item 
toute composante irréductible de $\M_{red}$ s'injecte dans $\GG \bc \M_{red}$
\item 
$\mathcal{V}\subset \GG\bc \M$ 
\end{itemize}
Puisque $\omega$ restreint à toutes les composantes irréductibles du schéma de type fini $\GG\bc \M_{red}$ est ample, $\omega$ est ample sur $\GG\bc \M_{red}$. On en déduit l'existence de $N\gg 0$ et d'une famille finie de sections $(s_i)_i$, $s_i\in H^0(\M,\omega^{\otimes N})^{\Gamma}$, telle que 
$$
\coprod_{\gamma\in \Gamma} \gamma.\mathcal{V} = \bigcup_i \{s_i\neq 0\}.
$$
Par hypothèse,
le $\O^+$-module localement libre $f^* (\omega^{\otimes N}\otimes_{\O_\M} \O_{\M_\eta}^+)$ 
s'étend canoniquement en un fibré en droites 
$$
\L:=\sp_* f^* (\omega^{\otimes N}\otimes_{\O_\M} \O_{\M_\eta}^+)
$$
sur $\X$. De même, la famille de sections $(f^* (s_i\otimes 1))_i$  s'étend en une famille de sections $(t_i)_i$ de $\L$. Considérons l'ouvert
$$
\mathcal{U}=\bigcup_i \{t_i\neq 0\}\subset \X.
$$
Par lissité formelle de $\X$, 
$$
\pi_0 (\mathcal{U}_\eta)= \pi_0(\mathcal{U}).
$$
On en déduit l'existence d'un ouvert/fermé $\mathcal{U}_0\subset \mathcal{U}$ tel que 
$$
g^{-1}(\mathcal{V})=\mathcal{U}_0.
$$
Cela conclut quant à la continuité de $g$. 
\end{proof}

\begin{exem}
Dans le théorème \ref{theo:minimal} supposons $\X$ $p$-adique quasi-compact lisse. Un morphisme 
$\X_\eta \drt \M_\eta$ est induit par un morphisme $$f:\widetilde{\X}\drt \M$$ où $\tilde{\X}\drt \X$ est un éclatement formel admissible au sens de Raynaud et on peut supposer $\widetilde{\X}$ normal. Le théorème \ref{theo:minimal} dit alors que  $f$ descend en un morphisme $\X\drt \M$ si et seulement si le fibré vectoriel $f^*\Lie H$ descend en un fibré vectoriel sur $\X$.
\end{exem}

\begin{rema}\label{rema:avant}
Le $\O^+$-module localement libre $\E^+:=\Lie (H)\otimes_{\O_{\M}} \O_{\M_\eta}^+$ est un objet subtile qui ne peut être retrouvé à partir des applications de périodes en fibré générique. En effet, si $\pi_{dR}:\M_\eta\drt \mathcal{F}$ est l'application des périodes de Hodge de-Rham alors $\E^+$ ne descend pas le long de $\pi_{dR}$ car il n'est pas Hecke équivariant contrairement à $\E^+\unp$. De même, on a une application de périodes de Hodge-Tate sur $\M_{\eta}$, un morphisme de $\O^+$-modules pro-étales $\a_{H^D}\otimes 1:T_p (H^{D})\otimes_{\Zp} \O_{\M_{\eta}}^+ \drt \E^{ +\vee}$ dont le conoyau est annulé par $p^{1/(p-1)}$. Mais cette application n'est pas surjective en générale et ne permet pas non plus de retrouver $\E^+$.
\end{rema}

\begin{rema}[suite de \ref{rema:avant}]\label{rema:avant 2}
Le choix d'un autre $\O^+$-module localement libre sur $\M_\eta$ que $\E^+$ dans $\E^+\unp$
\og donne lieu à d'autre modèles entiers\fg{} non formellement lisses. C'est par exemple le cas dans \cite{LivreIso} (sec. III.2.5, III.4.2 par exemple) en prenant l'image de $\a_{H^D}\otimes 1$ en niveau infini. De même, dans \cite{ScholzeTorsion} Scholze construit des modèles entiers de certaines variétés de Shimura en utilisant ce type d'éclatements/contractions.  
\end{rema}

Le corollaire suivant du théorème \ref{theo:minimal} mérite d'être cité. Il montre qu'on peut reconstruire le modèle entier $\M$ à partir de sa fibre générique $\M_\eta$ et d'une donnée \og linéaire\fg{}. Il s'agit essentiellement d'une traduction plus \og concrète\fg{} de \ref{theo:minimal}.

\begin{coro}\label{coro:caracterisation concrete}
Notons $\E=(\Lie\, H)^{rig}$ comme fibré vectoriel sur $\M_\eta$. Pour tout $K|\breve{\Q}_p$ de degré fini et $x\in \M_\eta (K)=\M(\O_K)$ notons $\La_x:=x^*\Lie\, H\subset x^*\E$ comme $\O_K$-réseau dans $x^* \E$. 
\begin{enumerate}
\item 
Le triplet 
$$
\big (\M_\eta, \E, (\La_x)_x \big )
$$
détermine complètement le modèle entier $\M$ de $\M_\eta$.
\item Si $\X$ est un morphisme $\breve{\Z}_p$-schéma formel localement formellement de type fini formellement lisse un morphisme $f:\X_\eta\drt \M_\eta$
s'étend en un morphisme $\X\drt \M$ si et seulement si $f^*\E$ est localement libre sur $|\X|$ et il existe des trivialisations locales sur des ouverts $\mathcal{U}$ recouvrant $\X$,
$$
u:\O_{\mathcal{U}_\eta}^m \iso (f^*\E)_{|\mathcal{U}_\eta},
$$
telles que pour tout $K|\breve{\Q}_p$ de degré fini, tout $x\in \mathcal{U}_\eta (K)=\mathcal{U}(\O_K)$,
$$
x^* u: \O_K^m\iso \La_{f(x)}.
$$
\end{enumerate}
\end{coro}

Bien sûr le théorème \ref{theo:minimal} s'étend immédiatement aux espaces de Rapoport-Zink associés à une donnée de type PEL non-ramifiée (\cite{RZ}). Plus généralement, soit $(G,b,\mu)$ une donnée de Shimura locale (\cite{RapoportViehmann}) non-ramifiée. On sait maintenant construire (\cite{ScholzeBerkeley}), pour le choix d'un sous-groupe compact hyperspécial $K\subset G(\Qp)$, un $\breve{\Q}_p$-espace rigide
$$
\mathrm{Sh}_K(G,b,\mu).
$$
Conjecturalement 
$$
\mathrm{Sh}_K(G,b,\mu) = \M_\eta
$$
où $\M$ est un $\breve{\Z}_p$-schéma formel localement formellement de type fini formellement lisse \og naturellement défini\fg{} à partir du modèle entier de $G$ associé à $K$. Une approche à ce problème est par exemple donnée dans \cite{BueltelPappas}.

\begin{conj}
Sous les hypothèses précédentes, il existe un $\O^+$-module localement libre de rang fini sur $\mathrm{Sh}_K(G,b,\mu)$ tel que la conclusion du théorème \ref{theo:minimal} s'applique.
\end{conj}

\begin{rema}
L'auteur ne sait malheureusement pas étendre le théorème \ref{theo:minimal} à des cas de variétés de Shimura locales en niveau parahorique non hyperspécial. Dans ces cas là, lorsqu'on sait les construire, les modèles entiers construits sont Cohen-Macauley  normaux. La géométrie des fibres de Milnor est alors plus compliquée et l'auteur ne connait pas d'analogue du théorème de Bartenwerfer qui permettrait de conclure.
\end{rema}

L'auteur ne sait pas répondre à la question suivante qui lui semble intéressante quant à la caractérisation des modèles entiers \og canoniques\fg{} des variétés de Shimura.

\begin{question}
Est-il possible d'établir un résultat du type \ref{theo:minimal} pour des variétés de Shimura de type PEL en niveau hyperspécial en $p$ ?
\end{question}

Enfin, notons le résultat suivant qui se déduit de la démonstration du théorème \ref{theo:minimal}.

\begin{theo}\label{theo:description des points}
Soit $\M$ un espace de Rapoport-Zink de type PEL non-ramifié sur $\spf (\breve{\Z}_p)$.
\begin{enumerate}
\item Pour un $n\geq 1$ et $K|\breve{\Q}_p$, pour $$f:\mathring{\BB}^n_K\drt \M_\eta,$$ sont équivalents:
\begin{enumerate}
\item il existe $x\in |\M|$ tel que $\text{Im}(f)\subset \mathrm{sp}^{-1}(x)$
\item le $\O^+$-module $f^*\E^+$ est trivial.
\end{enumerate}
\item L'application $x\mapsto \mathrm{sp}^{-1}(x)$ induit une bijection
$$
\M(\Fpb) \iso \{ B\subset \M_\eta \ \text{ ouvert } \ | \ \ B\simeq \mathring{\BB}^d_{\breve{\Q}_p}\text{ et } \E^+_{|B} \text{ est trivial }\}.
$$
\end{enumerate}
\end{theo}

Du point de vue du corollaire \ref{coro:caracterisation concrete} dire que $\E^+_{|B}$ est trivial est équivalent à dire que le fibré rigide analytique $\E_{|B}$ est trivial et il existe une trivialisation 
$$
u:\O_B^m \iso \E_{|B}
$$
telle que pour tout $x$ comme dans \ref{coro:caracterisation concrete}, $x\in B(K)$, 
$$
x^*u: K^m\iso x^*\E 
$$
induise un isomorphisme 
$$
\O_K^m\iso \Lambda_x.
$$

\bibliographystyle{plain}
\bibliography{biblio}
\end{document}